\documentclass{amsart}

\usepackage{amssymb}

\usepackage{amsfonts}
\usepackage{amssymb}
\usepackage{amsrefs}
\usepackage{color}
\usepackage{multirow}
\usepackage{tikz-cd}
\usepackage{float}

\usepackage{hyperref} 

\newtheorem{theorem}{Theorem}[section]
\newtheorem{lemma}[theorem]{Lemma}
\newtheorem{conjecture}[theorem]{Conjecture}
\newtheorem{corollary}[theorem]{Corollary}

\theoremstyle{definition}
\newtheorem{definition}[theorem]{Definition}

\theoremstyle{remark}
\newtheorem{remark}[theorem]{Remark}

\numberwithin{equation}{section}

\DeclareMathOperator{\Gal}{Gal}
\DeclareMathOperator{\GL2}{GL_2}

\DeclareMathOperator{\SL2}{SL_2}
\DeclareMathOperator{\Tr}{Tr}
\DeclareMathOperator{\Frob}{Frob}

\begin{document}

\title{On Fermat's equation over some quadratic imaginary number fields}


\author{George Turcas}
\address{Mathematics Institute\\
	University of Warwick\\
	Coventry\\
	CV4 7AL \\
	United Kingdom}
\email{g.c.turcas@warwick.ac.uk}
\thanks{The author is supported by
	an EPSRC PhD studentship. }


\date{}

\dedicatory{}

\begin{abstract}
	Assuming a deep but standard conjecture in the Langlands programme, we prove  Fermat's Last Theorem over $\mathbb Q(i)$.  Under the same assumption, we also prove that, for all prime exponents $p \geq 5$, Fermat's equation $a^p+b^p+c^p=0$ does not have non-trivial solutions over $\mathbb Q(\sqrt{-2})$ and $\mathbb Q(\sqrt{-7})$.
\end{abstract}

\maketitle

\section{Introduction}

 Wiles' remarkable proof of Fermat's Last Theorem inspired mathematicians to attack the Fermat equation over number fields via elliptic curves and modularity. Successful attempts over totally real quadratic fields had been carried out by Jarvis and Meekin \cite{jarvis2004fermat}, Freitas and Siksek \cite{FreitasSiksek}, \cite{AFreitasSiksek} and they rely on progress in modularity over totally real fields due to work of Barnett-Lamb, Breuil, Diamond, Gee, Geraghty, Kisin, Skinner,
Taylor, Wiles and others. In particular, modularity of elliptic curves over real quadratic fields was proved by Freitas, Le Hung and Siksek \cite{Freitas2015}.

On the other hand, modularity of elliptic curves over number fields with complex embeddings is highly conjectural. For general number fields, \c Seng\" un and Siksek \cite{ASiksek} proved an asymptotic version of Fermat's Last Theorem, under the assumption of two standard, but very deep conjectures in the Langlands programme.  They prove that for a number field $K$, satisfying some special properties, there exists a constant $B_K$, depending only on the field $K$, such that for all primes $p>B_K$, the equation
$$a^p + b^p + c^p = 0,$$
does not have solutions in $K \setminus \{0\}$.

The present work follows the skeleton of their article. Specialising to the fields $\mathbb Q(i), \mathbb Q(\sqrt{-2})$ and $\mathbb Q(\sqrt{-7})$, we were able to make their result effective and to obtain optimal bounds for the constant $B_K$. Combining this with previous work on the Fermat equation over number fields, was sufficient to derive Theorem \ref{main} below. Conjecture 4.1 in \cite{ASiksek} is known to hold for base-change newforms via the theory of base change functoriality, therefore we obtain a result that is dependant only on Serre's modularity conjecture (see Conjecture \ref{serreconj} below).

 Throughout this paper, given a number field $K$, we are going to denote by $\mathcal O_K$ its ring of integers and by $G_K=\Gal(\overline K/K)$ its absolute Galois group. We should emphasize some of the difficulties that had to be overcome.
\begin{enumerate}
\item	First, we had to obtain a small enough constant $B_K$ such that for $p> B_K$ the mod $p$ Galois representation that comes from a Frey curve is absolutely irreducible. This is treated asymptotically in \cite{ASiksek}, but we had to rely on upper bounds for possible prime torsion that can be achieved by  elliptic curves defined over number fields of small degree and class field theory computations to make the bound $B_K$ explicit. We should point out that the class field theory computations are different in an essential way from the ones that had been done in \cite{KrausQ} and \cite{FreitasSiksek}. The latter rely on the presence of non-trivial units in $\mathcal O_K$, which are not available in our setting.

\item To keep this introduction short, we defer the precise definitions of eigenforms and mod $p$ eigenforms to Section 2. Serre's modularity conjecture relates certain representations $G_K \to \GL2(\overline{\mathbb F}_p)$ to a mod $p$ eigenform of weight $2$ over $K$. In the classical situation Serre's modularity conjecture, now a theorem due to Khare and Wintenberger, relates some representations $G_{\mathbb Q} \to \GL2(\overline{\mathbb F}_p)$ to a complex eigenform of weight $2$ over $\mathbb Q$. The expression ``mod $p$" does not appear there, because it turns out that all mod $p$ eigenforms over $\mathbb Q$ are just reductions of complex eigenforms. This is not the case if $K$ is imaginary quadratic, but \c Seng\" un and Siksek \cite{ASiksek} managed to circumvent this difficulty by assuming that the prime $p$ is large enough. Unfortunately this step makes their result ineffective. For our choices of $K$, we  made this step effective by computing the torsion in certain cohomology groups using an algorithm of \c Seng\" un. This allowed us to lift the mod $p$ eigenforms to complex ones. 

\item A complex weight $2$ newform over $K$ with rational eigenvalues corresponds conjecturally to an elliptic curve or a fake elliptic curve defined over $K$ (see \cite{ASiksek}). The theory of base change functoriality shows that this conjecture holds for  the newforms that are base-change lifts of classical newforms. By explicit computation we find that all the newforms that we need to deal with are base-change lifts.

	\end{enumerate}

Our results assume a version of Serre's modularity conjecture (see Conjecture \ref{serreconj}) for odd, irreducible, continuous $2$-dimensional mod $p$ representations of $\Gal(\overline{\mathbb Q}/K)$ that are finite flat at every prime over $p$.

From now on, we restrict ourselves to $K \in \{ \mathbb Q(i), \mathbb Q(\sqrt{-2}), \mathbb Q(\sqrt{-7}) \}$ and we justify this restriction to the reader.

One reason is that we had to carry out some explicit computations in the cohomology groups of locally symmetric spaces (see Section 2) that were much simpler when the field had class number $1$.

	To carry out our argument and prove that there are no solutions to Fermat's equation, we had to prove that Galois representations associated to Frey curves are absolutely irreducible. In doing so, we made use of the fact that $K$ has primes of residual degree $1$ above $2$ so we end up with the three number fields above. On top of that, one should be aware of the existence of solutions $1^p +  (\omega)^p + (\omega^2)^p =0$, where $\omega$ is a primitive third root of unity  and $p$ is an odd prime. These are defined over $\mathbb Q(\sqrt{-3})$, making a proof of Fermat's Last Theorem over $\mathbb Q(\sqrt{-3})$ hopeless.
 
  \begin{theorem} \label{main} Assume Conjecture \ref{serreconj} holds for $K$. If $p \geq 5$ is a rational prime number, then the equation	
	\begin{equation} \label{maineq} a^p + b^p + c^p = 0 \end{equation}
has no solutions $a,b,c \in K\setminus \{ 0\}$.
\end{theorem}

Work is in progress to prove similar statements to Theorem \ref{main} for the other imaginary quadratic fields of class number one, but in those cases at the moment we are just able to prove that there are no solutions to \eqref{maineq} such that $Norm(abc)$ is even. We also hope to extend the methods to prove similar statements for other quadratic imaginary fields of non-trivial class number in the spirit of \cite{FreitasSiksek}.

\begin{corollary} \label{maincor} Assume Conjecture \ref{serreconj} holds for $\mathbb Q(i)$. Then, Fermat's Last Theorem holds over $\mathbb Q(i)$. In other words, for any integer $n \geq 3$, the equation
	$$a^n+b^n=c^n$$
has no solution $a,b,c \in \mathbb Q(i) \setminus \{ 0\}$. 
	
	\end{corollary}

The reader might now legitimately ask themselves why we were only able to derive this corollary for the case $K= \mathbb Q(i)$ and we will try to explain this now. In fact, our techniques prove the result claimed in Theorem \ref{main} for $p \geq 19$ when $K \in \{ \mathbb Q(i), \mathbb Q(\sqrt{-2}) \}$ and $p \geq 17$ when $K=\mathbb Q(\sqrt{-7})$ and we rely on previous work on Fermat's Last Theorem over number fields to conclude the same result for small values of $p$.
To be precise, a result of Gross and Rohrlich \cite{grossroh}*{Theorem 5.1} implies that there are no solutions to \eqref{maineq} for $p=5,7,11$ and our three choices of $K$. The case $p=13$ is covered completely by work of Tzermias \cites{tzermias}.

 Hao and Parry \cite{haoparry} also worked on Fermat's equation over quadratic fields and one of their results \cite{haoparry}*{Theorem 4} can be used to complete the remaining cases $p=17$ and $K \in \{ \mathbb Q(i), \mathbb Q(\sqrt{-2}) \}$ of our Theorem \ref{main}. All these works do not depend on the aforementioned conjecture of Serre.
 
 It is worth pointing out that if one knows that none of the equations $a^p +b^p + c^p =0$ for $p \geq 3$ prime and $a^4+b^4=c^4$ has a solution in $(a,b,c) \in (K\setminus \{0\})^3$, then one can deduce results analogous to Corollary \ref{maincor} for $K$. In fact that is how the corollary is proved.
 
 The equation $a^3+b^3+c^3=0$ describes a curve of genus $1$ and can be seen as an elliptic curve over $K$ after a choice of base-point. This elliptic curve has Mordell-Weil group isomorphic to $\mathbb Z/ 3 \mathbb Z$ when $K \in \{ \mathbb Q(i), \mathbb Q(\sqrt{-7}) \}$. The torsion points correspond to the trivial solutions (i.e. one of $a,b,c$ is $0$) to Fermat's equation under reparameterization. The situation is a little bit different when $K= \mathbb Q(\sqrt{-2})$, because the elliptic curve in play has rank $1$ and the non-torsion points on this curve correspond to infinitely many solutions $(a,b,c) \in \mathbb Q(\sqrt{-2})^{*}$ to $a^3+b^3+c^3=0$. All the computations here have been carried with the elliptic curve packages available in Magma \cite{magma}. For the equation $a^4+b^4=c^4$, we are going to refer to the work of Aigner \cite{aigner} who proved that the only non-trivial solutions to this equation over quadratic imaginary fields are defined over $\mathbb Q(\sqrt{-7})$. Note that $(1+ \sqrt{-7})^4+ (1-\sqrt{-7})^4 = 2^4$ is one of them. One can see now how the above discussion together with Theorem \ref{main} imply the corollary and also why such a corollary cannot be achieved for any of the other two fields that are covered in the theorem.
 
 \subsection*{Acknowledgements}
 
 We are indebted to the referees for their careful reading of this paper and for many helpful remarks. The author is very grateful to his advisor Samir Siksek for suggesting the problem and for his great support. It is also a pleasure to thank John Cremona and Haluk \c Seng\" un for useful discussions.

\section{Eigenforms for $\GL2$ over $K$ and Serre's modularity conjecture}

The exposition in this section follows closely along the lines of \cite{jonessen} and \cite{ASiksek}. A celebrated theorem of Khare and Wintenberger connects certain $2$-dimensional continuous representations of $G_{\mathbb Q}$ into $\GL2(\overline {\mathbb F}_p)$ with classical modular forms of the hyperbolic plane $\mathcal H_2$. In this section we are going to discuss a conjecture that aims to generalise the theorem of Khare and Wintenberger in a way that is going to be soon clarified.

Throughout this section $K$ is an imaginary quadratic field of class number $1$. Let $G=Res_{K/\mathbb Q}( \GL2)$ be the algebraic group over $\mathbb Q$ that is obtained from $\GL2$ by restriction of scalars from $K$ to $\mathbb Q$. The group of real points $G(\mathbb R) = \GL2(K \otimes_{\mathbb Q} \mathbb R) \cong \GL2(\mathbb C)$ acts transitively on $\mathcal H_3$. Fix an ideal $\mathfrak N \subset \mathcal O_K$ and consider the locally (adelic, see \cite{gunnells2014}) symmetric space
$$Y_0(\mathfrak N) = \Gamma_0(\mathfrak N) \backslash \mathcal H_3,$$
where $\Gamma_0(\mathfrak N)= \left\{  \left( \begin{array}{cc} a & b \\ c & d \end{array} \right) \in \GL2(\mathcal O_K) \mid c \in \mathfrak N \right\}$ is the usual congruence subgroup for the modular group $\GL2(\mathcal O_K)$.

Let $p$ be a prime that does not ramify in $K$ and is coprime to $\mathfrak N$. The cohomology groups $H^{i}(Y_0(\mathfrak N), \overline{\mathbb F}_p)$ come equipped with commutative Hecke algebras $\mathbb T_{\mathbb F_p}^i(\mathfrak N)$, generated by Hecke operators $T_{\mathfrak \pi}$ associated to the prime ideals $(\pi)$ of $\mathcal O_K$ away from $p \mathfrak N$. 

By a  mod $p$ eigenform $\Phi$ (over $K$) of level $\mathfrak N$ and degree $i$ we refer to a ring homomorphism $\Phi: \mathbb T_{\mathbb F_p}^i(\mathfrak N) \to \overline{\mathbb F}_p$.
It is known that the values of a mod $p$ eigenform generate a finite extension $\mathbb F$ of $\mathbb F_p$. Two such mod $p$ eigenforms of levels $\mathfrak N, \mathfrak M$ are said to be equivalent if their values agree on the Hecke operators associated to prime ideals away from $p \mathfrak N \mathfrak M$. 

Consider the complex cohomology groups $H^{i}(Y_0(\mathfrak N), \mathbb C)$. For all prime ideals $(\pi) \subset \mathcal O_K$, we can construct linear endomorphisms $T_{\pi}$ of $H^i(Y_0(\mathfrak N), \mathbb C)$, called Hecke operators. They form a commutative Hecke algebra denoted $\mathbb T^i_{\mathbb C}(\mathfrak N)$. 
By a complex eigenform $f$ over $K$ of degree $i$ and level $\mathfrak N$ we mean a ring homomorphism $f:\mathbb T^i_{\mathbb C}(\mathfrak N) \to \mathbb C$. The values of $f$ are algebraic integers and they generate a number field $\mathbb Q_f$. Two eigenforms $f,g$ are equivalent if their values agree on the Hecke operators. One should notice that since two equivalent eigenforms can be of different levels, the Hecke operators on which they agree might live in different Hecke algebras. A complex eigenform, say of level $\mathfrak N$, is called new if it is not equivalent to one whose level is a proper divisor of $\mathfrak N$. A complex eigenform $f$ is called trivial if $f(T_{\pi})=N_{K/\mathbb Q}(\pi) +1$ for all prime ideals $(\pi)$ of $\mathcal O_K$. In the setting of $\GL2$, non-triviality amounts to cuspidality.

We have to point out that the virtual cohomological dimension of $Y_0(\mathfrak N)$ is $2$ and for $p \geq 5$, the cuspidal parts of $H^1(Y_0(\mathfrak N), \overline{\mathbb F}_p)$ and $H^2(Y_0(\mathfrak N), \overline{\mathbb F}_p)$  are isomorphic as Hecke modules and the same holds for the cuspidal parts of $H^1(Y_0(\mathfrak N), \mathbb C))$ and $H^2(Y_0(\mathfrak N), \mathbb C)$. The latter follows from a duality result of Ash and Stevens \cite{ash1986}*{Lemma 1.4.3}. Therefore, in our setting, any mod $p$ or complex eigenform is equivalent to one of the same level and degree $1$. This particular case of a more general conjecture of Calegari and Emerton (see \cite{calegari2011completed}) holds in the present situation due to the low dimension.

Since everything happens in degrees $1$, $2$ and the non-trivial parts of the corresponding Hecke algebras are isomorphic, we are going to forget about the upper script $i$ and write $\mathbb T_{\mathbb F_p}(\mathfrak N)$ respectively $\mathbb T_{\mathbb C}(\mathfrak N)$, sometimes referring to the non-trivial part of the degree one Hecke algebra and sometimes referring to the one of degree two. The eigenforms in these algebras are sometimes called mod $p$ Bianchi modular forms, respectively Bianchi modular forms in the literature.

It is discussed in \cite{crewhit} how, via the Eichler-Shimura-Harder isomorphism, Bianchi modular forms can be analytically interpreted as vector valued real-analytic functions on the hyperbolic 3-space. We are going to be using this point of view in our examination of the Mellin transform of a Bianchi modular form in Section 5.

Given a complex eigenform $f$ and a prime $p$ that is coprime to the level of $f$, one can fix a prime ideal $\mathfrak p$ of $\mathbb Q_f$ that lies above $p$ and obtain a mod $p$ eigenform $\Phi_f$ (of the same level and degree) by setting $\Phi_f(T_{\pi})= f(T_{\pi}) \pmod{\mathfrak p} $ for all prime ideals $(\pi)$ coprime to $p\mathfrak N$. We say that a mod $p$ eigenform $\Phi$ lifts to a complex one if there is a complex eigenform $f$ with the same level and degree such that $\Phi=\Phi_f$.

Unlike the classical situation in which $K=\mathbb Q$, when $K$ is quadratic imaginary not all the mod $p$ eigenforms lift to complex ones. To see this, consider the following short exact sequence given by multiplication-by-$p$
\begin{center}
\begin{tikzcd}	
	0 \arrow[r] & \mathbb Z \arrow[r, "\times p"]  & \mathbb Z \arrow[r]  & \mathbb F_p \arrow[r] & 0
\end{tikzcd}.
	\end{center}
This gives rise to a long exact sequence on cohomology
\begin{center}
	\begin{tikzcd}
		\dots H^1(Y_0(\mathfrak N), \mathbb Z) \arrow[r,"\times p"] & H^1(Y_0(\mathfrak N), \mathbb Z) \arrow[r] \arrow[d, phantom, ""{coordinate, name=Z}] & H^1(Y_0(\mathfrak N), \mathbb F_p) \arrow[dll,
		"\delta",
		rounded corners,
		to path={ -- ([xshift=2ex]\tikztostart.east)
			|- (Z) [near end]\tikztonodes
			-| ([xshift=-2ex]\tikztotarget.west)
			-- (\tikztotarget)}] \\ H^2(Y_0(\mathfrak N), \mathbb Z) \arrow[r] & \dots
		\end{tikzcd}
\end{center}
from which we can extract the short exact sequence
\begin{center}
\begin{tikzcd}[column sep = small]
	0 \arrow[r] & H^1(Y_0(\mathfrak N), \mathbb Z) \otimes \mathbb F_p \arrow[r] & H^1(Y_0(\mathfrak N), \mathbb F_p) \arrow[r, "\delta"] & H^{2}(Y_0(\mathfrak N), \mathbb Z)[p] \arrow[r] & 0
	\end{tikzcd}.	
	
	\end{center}

In the above, the presence of $p$-torsion in $H^2(Y_0(\mathfrak N), \mathbb Z)$ is the obstruction to surjectivity in the reduction map. The existence of an eigenform (complex or mod $p$) is equivalent to the existence of a class in the corresponding cohomology group that is a simultaneous eigenvector for the Hecke operators such that its eigenvalues match the values of the eigenform. With this interpretation, we can deduce that every mod $p$ eigenform of degree $1$ lifts to a complex one when $H^{2}(Y_0(\mathfrak N), \mathbb Z)$ has no $p$-torsion.

We will be using a special case of Serre's modularity conjecture over number fields. In his landmark paper \cite{Serre}, Serre  conjectured that all absolutely irreducible, odd mod $p$ Galois representation of $G_{\mathbb Q}=\Gal(\overline{\mathbb Q}/ \mathbb Q)$ \textit{arise} from a cuspidal eigenform $f$. Here $f$ is a classical cuspidal modular form. In the same article, Serre gave a recipe for the level $N$ and the weight $k$ of the sought after eigenform. As previously mentioned, this conjecture was proved by Khare and Wintenberger \cite{Khare2009}.

We are going to state a conjecture which concerns mod $p$ representations of $G_K$, the absolute Galois group of the quadratic imaginary $K$. 

\begin{conjecture} \label{serreconj} Let $\overline \rho : G_K \to \GL2(\overline{\mathbb F}_p)$ be an irreducible, continuous representation with Serre conductor $\mathfrak N$ (prime-to-$p$ part of its Artin conductor) and trivial character (prime-to-$p$ part of $\det(\overline \rho)$). Assume that $p$ is unramified in $K$ and that $\overline \rho|_{G_{K_{\mathfrak p}}}$ arises from a finite-flat group scheme over $\mathcal O_{K_{\mathfrak p}}$ for every prime $\mathfrak p|p$. Then there is a mod $p$ eigenform $\Phi : \mathbb T_{\mathbb F_p}\left(Y_0(\mathfrak N) \right) \to \overline{\mathbb F}_p$ such that for all prime ideals $(\pi) \subseteq \mathcal O_K$, coprime to $p \mathfrak N$
	$$Tr ( \overline \rho( \Frob_{ (\pi) }) ) = \Phi(T_{\pi}).$$
\end{conjecture}

\begin{remark}
	When stated for more general number fields, this conjecture restricts to odd representations. A representation is odd if the determinant of every complex conjugation is $-1$, but since our $K$ is totally complex $G_K$ does not contain any complex conjugations and we will regard every mod $p$ representation of $G_K$ automatically as odd. 
	\end{remark}

Although it is conjecturally easy to predict the level $\mathfrak N$ of such an eigenform, doing the same thing for the weight can be very difficult. A quite involved general weight recipe for $\GL2$ over number fields was given by Gee, Herzig and Savitt in \cite{GeeHerzSav}. We will just mention that this recipe depends on the restriction $\overline{\rho}|_{I_{\mathfrak p}}$ to the inertia subgroups for the primes $\mathfrak p \subset \mathcal O_K$ above $p$.  We only considered very special representations $\overline{\rho}$ (that are finite flat at $\mathfrak p|p$), for which Serre's original weight recipe applies and predicts the trivial weight \cite{Serre}. This is why we end up with classes in $H^1(Y_0(\mathfrak N),  \overline{\mathbb F}_p)$, the trivial weight meaning that we get $\overline{\mathbb F}_p$ as coefficient module.

\section{Fermat equation with exponent $p$ and the Frey Curve}

Let us fix some notation:

$p$ - a rational prime number;

$K$ is one of $\mathbb Q(i)$, $\mathbb Q(\sqrt{-2})$ or $\mathbb Q(\sqrt{-7})$;

$\mathcal O_K$ - the ring of integers of $K$;

$S$ - set of prime ideals of $\mathcal O_K$ that lie above the prime $2$. 

\vspace{0.5cm}

By the Fermat equation with exponent $p$ over $K$, we mean
\begin{equation} \label{fermeq}
a^p + b^p + c^p=0, \hspace{1cm} a,b,c \in \mathcal O_K.
\end{equation}

We say that a solution $(a,b,c) \in \mathcal O_K^3$ is trivial if $abc = 0$ and non-trivial otherwise. We shall henceforth assume that $p \geq 17.$ Note that any solution $(a,b,c)$ of \eqref{main} satisfying $abc \neq 0$ can be scaled such that $a,b,c$ become integral and the triple gives a non-trivial solution to \eqref{fermeq}.

 The rational prime $2$ is ramified in the first choices of $K$ and completely split in $\mathbb Q(\sqrt{-7})$. In particular, the residue field of the primes in $S$ is always $\mathbb F_2$, and this is essential in our argument. Let $(a,b,c) \in \mathcal O_K^3$ be a non-trivial solution to the Fermat equation \eqref{fermeq}. Since the class number of $K$ is $1$, we can scale the solution such that $a,b$ and $c$ are coprime. Associated to $(a,b,c)$ is the Frey curve
\begin{equation} \label{freycurve}
E=E_{a,b,c}: Y^2 = X(X-a^p)(X+b^p).
\end{equation}

We write $\overline{\rho}=\overline{\rho}_{E,p}$ for the residual Galois representation
$$\overline{\rho}_{E,p}: G_K \to Aut(E[p]) \cong \GL2(\mathbb F_p )$$
induced by the action of $G_K$ on the $p$-torsion of $E[p]$. 

From Lemma 3.7 of \cite{AFreitasSiksek} it follows that $E$ has potentially multiplicative reduction at $\mathfrak a \in S$ and that $p \mid \# \overline{\rho}(I_{\mathfrak a})$, where $I_{\mathfrak a}$ denotes the inertia subgroup of $G_K$ at $\mathfrak a$. This information is crucial for everything that follows.

The following is a particular case of \cite{ASiksek}*{Lemma 5.4}. Since it is central to our discussion, we include a proof here.

\begin{lemma} \label{semistab} The Frey curve $E$ is semistable away from $S$. The determinant of $\overline{\rho}$ is the mod $p$ cyclotomic character. Its Serre conductor $\mathcal{N}$ is supported on $S$ and belongs to a finite set that depends only on the field $K$. The representation $\overline{\rho}$ is finite flat at every prime $\mathfrak p$ of $K$ that lies above $p$.
	\end{lemma}

\begin{proof}
	Let $c_4$ and $\Delta$ denote the usual invariants of the model $E$ given in \eqref{freycurve}. These are given by the formulae
	$$c_4= 2^{4}(b^{2p}-a^pc^p) \text{ and } \Delta = 2^4 (abc)^{2p}.$$
	If a prime ideal $\mathfrak q$ of $\mathcal O_K$ divides both $c_4$ and $\Delta$, the equations above together with $a^p + b^p + c^p =0$ imply that $\mathfrak q$ divides $2$, so $\mathfrak q \in S$. 
	
	Let $\mathfrak q \not \in S$ be a prime of $\mathcal O_K$.
	By well-known results \cite{silver1}*{Section VII}, one deduces that the model \eqref{freycurve} for $E$ is minimal and the curve is semistable at $\mathfrak q$.
	
	Moreover, $p \mid v_{\mathfrak q}(\Delta)$. It follows from \cite{Serre} that $\overline{\rho}$ is finite flat at $\mathfrak q$ if $\mathfrak q$ lies above $p$. From the results in the same article, we can deduce that $\overline{\rho}$ is unramified at $\mathfrak q$ if $\mathfrak q \nmid p$. Since the Serre conductor $\mathcal N$ is by definition not supported on primes above $p$, we get that it is actually supported only on the primes in $S$. We also know that $\mathcal N$ divides the conductor of $E$, therefore we can bound the exponent of $\mathfrak q $ in $\mathcal N$ using \cite{silver2}*{Theorem IV.10.4}. We get
	$$ v_{\mathfrak q}(\mathcal N) \leq 2 + 3v_{\mathfrak q}(3)+ 6 v_{\mathfrak q}(2)=2+6v_{\mathfrak q}(2),$$
	for all $\mathfrak q \in S$.
	
The statement concerning the determinant is a well known consequence of the theory of the Weil pairing.
	
%
%
	
	\end{proof}

We denote by $N_E$ the conductor of $E$ and apply Lemma 4.4 in  \cite{FreitasSiksek} to this elliptic curve. Despite the fact that the results of \cite{FreitasSiksek} are stated for real quadratic fields, the arguments of this specific lemma are purely local and they continue to apply in our setting. We describe this process in details for $K=\mathbb Q(i)$ remarking that the computation is analogous for the other two number fields. Following the notation in \cite{FreitasSiksek},  one can compute values of $\lambda_i \in \mathcal O_K \setminus \mathfrak b$, which are some cokernel representatives of the map $\Phi:\mathcal O_K^{*} \to \mathcal (O_K/ \mathfrak b)^{*}/ \left((\mathcal O_K / \mathfrak b)^{*}\right)^2$. 

When $K=\mathbb Q(i)$, let $2 \mathcal O_K =\mathfrak a^2$. In the notation of \cite{FreitasSiksek}, this gives $\mathfrak b= \mathfrak a^5$, an ideal of norm $32$. The image of $\Phi:\mathcal O_K^{*} \to \mathcal (O_K/ \mathfrak b)^{*}/ \left((\mathcal O_K / \mathfrak b)^{*}\right)^2$ has order $2$ and the codomain $(O_K/ \mathfrak b)^{*}/ \left((\mathcal O_K / \mathfrak b)^{*}\right)^2$ is isomorphic to $\left(\mathbb Z/ 2 \mathbb Z \right)^3$. One can take as representatives of the cokernel of $\Phi$ the elements $\lambda_1 = 1$, $\lambda_2 =2+i$, $\lambda_3 = -3$ and $\lambda_4=-2+i$ in $\mathcal O_K$. This gives $\max\limits_{1 \leq i \leq 4} v_{\mathfrak a}(\Delta(K_{\mathfrak a}(\sqrt{\lambda_i})/K_{\mathfrak a}))=4$. Even though one could do all the above computations using pen and paper only, the author made use of basic features of Magma at some steps.

 After performing analogous computations in $\mathbb Q(\sqrt{-2})$ and $\mathbb Q(\sqrt{-7})$, one can see that a consequence of Lemma 4.4 in loc. cit. is the following. We can scale the triple $(a,b,c)$ by a unit in $\mathcal O_K^{*}$ such that we find ourselves in one of the cases listed in Table \ref{tablecond}.
 
\begin{table}[h]
	\centering
	\caption{}\label{tablecond}
	\begin{tabular}{|c|c|c|}
		\hline
		\textbf{Number field $K$}         & \textbf{Factorization of $2\mathcal O_K$} & \textbf{Valuation of $N_E$} \\ \hline
		$\mathbb Q(i)$  & $2\mathcal O_K = \mathfrak a^2$        & $v_{\mathfrak a}(N_E) = 8$           \\ \hline
		$\mathbb Q(\sqrt{-2})$ &   $2\mathcal O_K = \mathfrak a^2$       &   $v_{\mathfrak a}(N_E) = 8$         \\ \hline
		\multirow{2}{*}{ $\mathbb Q(\sqrt{-7})$} &  \multirow{2}{*}{$2\mathcal O_K = \mathfrak a_1 \cdot \mathfrak a_2, \, \mathfrak a_1 \neq \mathfrak a_2$  }      &    $v_{\mathfrak a_1}(N_E)=4 \text{ and } v_{\mathfrak a_2}(N_E)=1$       \\ \cline{3-3} 
		&           &   $v_{\mathfrak a_1}(N_E)=1 \text{ and } v_{\mathfrak a_2}(N_E)=4$        \\ \hline
	\end{tabular}
\end{table}

\section{Local computations and irreducibility of $\overline{\rho}$}

For applying Conjecture \ref{serreconj} to the Galois representation $\overline{\rho}= \overline{\rho}_{E,p}$, we have to prove that it is absolutely irreducible. We first prove irreducibility and the stronger condition will follow. 

The conductor exponent of an elliptic curve is defined (see \cite{silver2}*{Chapter IV}) as the sum between a tame and a wild part. Similarly, the conductor exponent of a Galois representation can be written as the sum between a tame and a wild part. In the following lemma, we think of a mod $p$ character as a $1$ dimensional representation over $\mathbb F_p$ and we define the wild part of its conductor exponent accordingly.

\begin{lemma} \label{conval} Let $E$ be an elliptic curve of conductor $\mathcal N$ defined over a number field $K$ and $p$ a rational prime. Suppose $\overline{\rho}_{E,p}$ is reducible, that is,
$$\overline{\rho}_{E,p} \sim \left( \begin{array}{cc} \theta & \star \\ 0 & \theta' \end{array} \right) \text{ with } \theta, \theta': G_K \to \mathbb F_p^{*}  \text{ satisfying } \theta\theta' = \chi_p,$$
where $\chi_p$ is the mod $p$ cyclotomic character. Let $\mathfrak q \nmid p$ be a prime in $K$ of additive reduction for $E$. Then $\delta_{\mathfrak q}(\mathcal N)$ is even and 
$$\delta_{\mathfrak q}(\theta) = \delta_{\mathfrak q}(\theta') = \frac{\delta_{\mathfrak q}(\mathcal N)}{2},$$
where $\delta_{\mathfrak q}(\theta)$, $\delta_{\mathfrak q}(\theta')$ and $\delta_{\mathfrak q}(N)$ are the wild parts of the exponent at $\mathfrak q$ in the conductors of $\theta$, $\theta'$ and $E$.
\end{lemma}

\begin{proof}
	The first equality follows from the fact that $\theta \theta' = \chi_{p}$, therefore when restricted to the absolute inertia group $I_{\mathfrak q}$, the characters are inverses of each other.  
	
	In what follows, we are going to think of $\theta$ and $\theta'$ as one dimensional representations from $G_K$ targeting the one dimensional subspaces of $E[p]$ on which we see their actions.

Denote by  $K_{\mathfrak q}$ the completion of the number field $K$ at the prime $\mathfrak q$ and by $L=K_{\mathfrak q}(E[p])$ the extension of $K_{\mathfrak q}$ that we get by adjoining the coordinates of the $p$-torsion of $E$.
	
	As one can see in \cite{silver2}*{chapter IV}, the wild part of the exponent of $\mathcal N$ at $\mathfrak q$ can be computed as
	
	$$\sum_{i\geq 1} \frac{g_i(L/K_{\mathfrak q})}{g_0(L/K_{\mathfrak q}) } \cdot \dim_{\mathbb F_p} \left( E[p]/ E[p]^{G_i(L/K_{\mathfrak q})}  \right),$$
	where $G_0(L/K)$ is the inertia group and $G_i(L/K_{\mathfrak q})$ the $i$-th ramification group of of $L/K_{\mathfrak q}$. The quantities $g_0(L/K_{\mathfrak q})$ and $g_i(L/K_{\mathfrak q})$ are the orders of the aforementioned groups. The $i$-th inertia subgroup of $L/K_{\mathfrak q}$ can be identified with the $i$-th inertia subgroup of $\mathfrak q$ in the $G_K$ for all $i \geq 0$.

	It is known that $G_1(L/K_{\mathfrak q})$ is a Sylow group of order coprime to $p$, which implies that when restricted to the wild inertia $G_1(L/K_{\mathfrak q})$, the representation $\overline \rho$ becomes
	$$\overline \rho|_{G_1(L/K_{\mathfrak q})} \sim \left( \begin{array}{cc} \theta & 0 \\ 0 & \theta' \end{array} \right).$$
	In fact, since $\mathfrak q \nmid p$, the cyclotomic character $\chi_p = \theta \theta'$ is trivial on $G_0(L/K_{\mathfrak q})$, we see that
	$$ \overline \rho|_{G_1(L/K_{\mathfrak q})} \sim \left( \begin{array}{cc} \theta & 0 \\ 0 & \theta^{-1} \end{array} \right).$$
	
	When restricted to $G_1(L/K_{\mathfrak q})$, let $V_1$ be the subspace of $E[p]$ on which $\theta$ acts and $V_2$ the subpace of $E[p]$ on which we see the action of $\theta'$. One can see that $E[p] / E[p]^{G_i(L/K_{\mathfrak q})} \equiv V_1/V_1^{G_i(L/K_{\mathfrak q}) } \oplus V_2/V_2^{G_i(L/K_{\mathfrak q})} $
	and $\dim_{\mathbb F_p} V_1/ V_1^{G_i(L/K_{\mathfrak q})} = \dim_{\mathbb F_p} V_2/V_2^{G_i(L/K_{\mathfrak q})}$ for all $i \geq 1$.
	
	If we think of $\theta$ as a 1 dimensional representation from $\Gal(K(E[p])/K) \to V_1$, then the definition of the wild part of its conductor at $\mathfrak q$ is
	$$\delta_{\mathfrak q}(\theta)= \sum_{i \geq 1} \frac{}{} \dim_{\mathbb F_p} \left( V_1/V_1^{G_i(F/K)} \right),$$
	where $V_1$ is the subspace of $E[p]$ on which $\overline \rho|_{G_1(F/K)}$ acts as $\theta$. Using the analogous formula for $\delta_{\mathfrak q}(\theta')$, the conclusion follows.
	
\end{proof}

\begin{theorem}
Let $p \geq 19$ if $K \in \{\mathbb Q(i) , \mathbb Q(\sqrt{-2}) \}$ and $p \geq 17$ if $K = \mathbb Q(\sqrt{-7})$. Then $\overline{\rho}$ is irreducible.	
	\end{theorem}

\begin{proof}
	If $\overline {\rho}_{E,p}$ is reducible, then we can write
	$$\overline{\rho}_{E,p} \sim \left( \begin{array}{cc} \theta & * \\ 0 & \theta' \end{array}  \right),$$
	where $\theta$ and $\theta'$ are characters $G_K \to \mathbb F_p^{*}$, and $\theta \theta' = \chi_{p}$, the mod $p$ cyclotomic character given by the action of $G_K$ on the group $\mu_p$ of $p^{\text{th}}$ roots of unity. 
	
	
	From the proof of Lemma 1 in \cite{KrausQ} we know that $\theta, \theta'$ are unramified away from $p$ and the primes of additive reduction of $E$, i.e. the primes above $2$ in $\mathcal O_K$. Using the notation introduced in Lemma \ref{conval},  we have that
	$$\delta_{\mathfrak a}(\mathcal N_{\theta}) = \delta_{\mathfrak a}(\mathcal N_{\theta'}) = \frac{1}{2} \delta_{\mathfrak a}(N_E),$$
	where $\mathfrak a$ is an ideal of $\mathcal O_K$ above $2$.
	
The tame part of the exponent of a prime in the conductor of a character is at most $1$ and must be equal $1$ if the wild part is non-zero. On the other hand, the corresponding quantity in the conductor of an elliptic curve is at most $2$ and must be equal to $2$ if its wild part is non-zero. For the precise definitions, see  \cite{silver2}*{chapter IV}.
	
(i) Suppose $p$ is coprime to $\mathcal N_{\theta}$ or $\mathcal N_{\theta'}$. Since the conductor of an elliptic curve is isogeny invariant, by eventually replacing $E$ with the $p$-isogenous curve $E/\langle \theta \rangle$ we can assume that $p$ is coprime to $N_{\theta}$ and hence deduce that $\theta$ is unramified away from the primes in $S$. 

For $K=\mathbb Q(\sqrt{-7})$, denote by $\mathfrak a_1, \mathfrak a_2$ the prime ideals of $K$ above $2$. From the table above, we see that
$$N_{\theta} \in \{ \mathfrak a_{1}^2, \mathfrak a_2^2 \}.$$
Therefore, $\theta$ is a character of the ray class group of modulus $\mathfrak a_1^2$ or $\mathfrak a_2^2$. Both of these ray class groups are trivial, which implies that $\theta$ has order $1$ and hence that $E$ has a point of order $p$ defined over $K$. Our assumption $p \geq 17$ contradicts the results of Theorem 3.1 in \cite{Kamienny1992} which implies the order of such a torsion point is less than or equal to $13$.

We threat the cases $K=\mathbb Q(i), \mathbb Q(\sqrt{-2})$ together. Here we assumed that $p \geq 19$. Using the values computed in the previous table and Lemma \ref{conval}, we see that the only possibility for the conductor of $\theta$ is
$$N_{\theta} = \mathfrak a^4,$$
where $\mathfrak a$ is the unique prime above $2$ in $\mathcal O_K$. The ray class groups for these fields and modulus are $\mathbb Z/ 2 \mathbb Z$ and $\mathbb Z/4 \mathbb Z$ respectively. These computations were done using Magma.

In turn, for $\mathbb Q(i)$ and $\mathbb Q(\sqrt{-2})$,  $\theta$ is a character of that corresponding ray class group, therefore the order of $\theta$ divides the exponent of the group. If $\theta$ has order $1$ then $E$ has a point of order $p$ over $K$ and we get a contradiction exactly as before. Similarly, if the order of $\theta$ is $2$ then $E$ has a $p$-torsion point defined over a quadratic extension $L$ of $K$. But $[L:\mathbb Q] = [L:K] \cdot [K:\mathbb Q]=4$ and the possible prime torsion of elliptic curves over quartic fields determined by Derickx, Kamienny Stein and Stoll \cite{Derkamstoll} would imply that $p \leq 17$, a contradiction.

If $\theta$ has order $4$, then let $L$ be the quadratic extension of $K$ that is cut by the character $\theta^2$. The restriction $\phi=\theta|_{G_L}$ is a quadratic character of $G_L$ and therefore the twist by $\phi$ of $E$, regarded as an elliptic curve over $L$, is an elliptic curve with a $p$-torsion point defined over $L$. The field $L$ is of total degree $4$ over $\mathbb Q$ and hence we get a contradiction as in the other case.

(ii) Suppose now that $p$ is not coprime with $\mathcal N_{\theta}$ nor with $\mathcal N_{\theta'}$. Since $p \geq 19$, we know that $p$ is not ramified in $K$. If we suppose that $p$ is inert in $K$, then from \cite{KrausQ}*{Lemme 1} it follows that $p$ divides only one of the conductors, a contradiction with the hypothesis.

The only case that has to be considered now is when $p$ splits in $K$. Let $\mathfrak p_1, \mathfrak p_2$ be the two ideals of $\mathcal O_K$ such that $p \mathcal O_K = \mathfrak p_1 \mathfrak p_2$. We can suppose that $\mathfrak p_1 | \mathcal N_{\theta}$, $\mathfrak p_1 \nmid \mathcal N_{\theta'}$ and $\mathfrak p_2 | \mathcal N_{\theta'}$, $\mathfrak p_2 \nmid \mathcal N_{\theta}$. The primes $\mathfrak p_1$, $\mathfrak p_2$ are unramified so it follows from \cite{serre72}*{Proposition 12} that $E$ has good ordinary or multiplicative reduction at these primes and $\theta|_{I_{\mathfrak p_1}} = \chi_p|_{I_{\mathfrak p_1}}$ and $\theta'|_{I_{\mathfrak p_2}} = \chi_p|_{I_{\mathfrak p_2}}$. 

The character $\theta^2$ is unramified everywhere except $\mathfrak p_1$, because all the bad places for $E$ are of potentially multiplicative reduction. We also know that $\theta^2|_{ I_{\mathfrak p_1}} = \chi^2_p|_{I_{\mathfrak p_1}}$. 

From Lemma \ref{classfieldlemma} below, it follows that
  $$\theta^2(\sigma_{\mathfrak a}) \equiv N_{K_{\mathfrak p_1}/\mathbb Q_p}\left( \iota_{\mathfrak p_1} (\mathfrak a) \right)^2 \pmod{p},$$
 where $\mathfrak a$ is a prime ideal of $\mathcal O_K$ that lies above $2$ (principal, of course) and $\sigma_{\mathfrak a}$ is the Frobenius element at $\mathfrak a$.
 
 As explained in \cite{ASiksek}*{Lemma 6.3}, by eventually replacing $E$ with the $p$-isogenous curve $E/ \langle \theta \rangle$ we can assume that $\theta^2(\sigma_{\mathfrak a}) \equiv 1 \pmod{p}$. We have that $$N_{K_{\mathfrak p_1}/\mathbb Q_p}\left( \iota_{\mathfrak p_1} (\mathfrak a) \right)^2 -1 = 3,$$ so $p \mid 3$ which gives a contradiction.

\end{proof}

We have used above the following class field theory lemma. We do not include a proof here since that would be {\it mutatis mutandis}, just a specialisation of \cite{David}*{Proposition 2.4}.

\begin{lemma} \label{classfieldlemma}
Let $(p)= \mathfrak p_1 \mathfrak p_2$ be a rational prime that splits in $K$ and let $\theta:G_K \to \mathbb F_p^{*}$, be a character with the property that $\theta^2$ is unramified everywhere except $\mathfrak p_1$.
Then, for any prime to $p$ element $0 \neq \alpha \in K$, the following congruence relation holds
$$ \prod_{\mathfrak q \nmid p} \theta^2(\sigma_{\mathfrak q})^{v_{\mathfrak q}(\alpha)} \equiv N_{K_{\mathfrak p_1}/ \mathbb Q_p} \left( \iota_{\mathfrak p_1} ( \alpha) \right)^2 \pmod{p}.$$
\end{lemma}

\begin{remark} For every place $v$, the map $\iota_{v}$ is the inclusion of $K$ into the completion $K_v$. Abusing notation we write $i_{\mathfrak q}$ and $K_{\mathfrak q}$ for the inclusion, respectively the completion of $K$, with respect to the corresponding prime ideal $\mathfrak q$.
	\end{remark}

We obtain the following corollary, which implies absolute irreducibility.

\begin{corollary} For $p \geq 19$, the Galois representation $\overline{\rho}$ is surjective.	
	\end{corollary}

\begin{proof} In the proof of \cite{AFreitasSiksek}*{Lemma 3.7}, the authors used that $v_{\mathfrak a}(j(E)) <0$ and $p \nmid v_{\mathfrak a}(j(E))$ to deduce that $E$ has multiplicative reduction at $\mathfrak a \in S$ and that the cardinality of $\overline{\rho}(I_{\mathfrak a})$ is also divisible by $p$, where $I_{\mathfrak a}$ is the inertia subgroup at $\mathfrak a$. Since all the  irreducible subgroups of $\GL2(\mathbb F_p)$ that contain an element of order $p$ have $\SL2(\mathbb F_p)$ as a subgroup, we get that $\SL2(\mathbb F_p) \subseteq \overline{\rho}(G_K)$. The determinant of $\overline\rho$ is the mod $p$ cyclotomic character $\chi_p$, and since $K \cap \mathbb Q(\zeta_p) = \mathbb Q$, the latter is surjective. All the above imply that $\overline{\rho}(G_K)= \GL2(\mathbb F_p)$.
	
\end{proof}

\section{Applying Serre's conjecture}

We are now making use of the Conjecture \ref{serreconj}. This predicts the existence of a mod $p$  eigenform $\Phi : \mathbb T_{\overline{\mathbb F}_p}(\mathcal N) \to \overline{\mathbb F}_p$ over $K$ such that for every prime ideal $(\pi) \subset \mathcal O_K$, coprime to $p \mathcal N$. 
$$Tr\left(\overline{\rho}(\Frob_{(\pi)}) \right) = \Phi(T_{\pi}),$$
where $T_{\pi}$ is a Hecke operator. 

The trace elements $Tr\left(\overline{\rho}(\Frob_{(\pi)}) \right)$ lie in $\mathbb F_p$, therefore $\Phi$ corresponds to a class in $H^1(Y_0(\mathcal N), \mathbb F_p)$ that is an eigenvector for all such Hecke operators $T_{\pi}$.  

As previously described, the obstruction in lifting such mod $p$ eigenforms to complex eigenforms is given by the presence of $p$-torsion in $H^2(Y_0(\mathcal N), \mathbb Z)$. It is known (see for example \cite{ash1986}*{page 202}) that if the last common multiple of the orders of elements of finite order in $\Gamma_0(\mathcal N)$ is invertible in the coefficients module, then simplicial cohomology and group cohomology are the same, in other words $H^2(Y_0(\mathcal N), \mathbb Z[\frac{1}{6}]) \cong H^2(\Gamma_0(\mathcal N), \mathbb Z[\frac 1 6])$.

Lefschetz duality for cohomology with compact support \cite{sentors}*{Section 2} gives a relation between the first homology and the second cohomology
$H_{1}(\Gamma_0(\mathcal N), \mathbb Z\left[\frac{1}{6}\right]) \cong H^2(\Gamma_0(\mathcal N) , \mathbb Z \left[\frac{1}{6}\right])$. It is also known that the abelianization $\Gamma_0(\mathcal N)^{ab} \cong H_{1}(\Gamma_0(\mathcal N), \mathbb Z)$ and therefore, for primes $p >3$, if the group $H^2(Y_0(\mathcal N) , \mathbb Z)$ has a $p$-torsion element, then $\Gamma_0(\mathcal N)^{ab}$ will have a $p$-torsion as well.

Using the Magma implementation, kindly provided to us by Haluk \c Seng\" un, of his algorithm in \cite{sentors} we computed the abelianizations $\Gamma_0(\mathcal N)^{ab}$. This algorithm uses as input generators and relations for $\GL2(\mathcal O_K)$ which were computed by Swan in \cite{swan71}. The reader can consult the Magma code at:
\begin{center}
 \url{https://warwick.ac.uk/fac/sci/maths/people/staff/turcas/fermatprog}.
\end{center}

We list the torsion found in Table \ref{primetorsion1}, where $\mathfrak a, \mathfrak a_1$ and $\mathfrak a_2$ are primes above $2$ in the corresponding number fields.
\begin{center}
\begin{table}[h]
	\centering
	\caption{prime torsion in $\Gamma_0(\mathcal N)^{ab}$}
	\label{primetorsion1}
	\begin{tabular}{|c|c|l|}
		\hline
		\textbf{Number field}         & \textbf{Level $\mathcal N$} & \textbf{primes $l$ such that $\Gamma_0(\mathcal N)^{ab}[l] \neq 0$ } \\ \hline
		$\mathbb Q(i)$&      $\mathfrak a^8$      &  2          \\ \hline
		$\mathbb Q(\sqrt{-2})$&  $\mathfrak a^8$         & 2           \\ \hline
		\multirow{2}{*}{$\mathbb Q(\sqrt{-7})$} & $\mathfrak a_1^4 \mathfrak a_2$          &  2,3          \\ \cline{2-3} 
		&   $\mathfrak a_1 \mathfrak a_2^4$        &   2,3        \\ \hline
	\end{tabular}
\end{table}
\end{center}

Since we have chosen $p \geq 19$ and there is no $p$-torsion in the subgroups of interest, the mod $p$ eigenforms must lift to complex ones. Using the available Magma package for Bianchi modular forms, we compute these spaces of eigenforms. The implementation, due to Dan Yasaki, is based on an algorithm of Gunnels \cite{gun99} for computing the Voronoi polyhedron and provides a replacement for the modular symbol algorithm used by Cremona to compute the action of the Hecke operators in \cite{cremona84}.

The dimension of the respective cuspidal spaces for $\mathbb Q(i), \mathbb Q(\sqrt{-7})$ is $0$ and the dimension of the cuspidal space of level $\mathfrak a^8$ for $\mathbb Q(\sqrt{-2})$ is $6$.

Since there are no eigenforms at the predicted levels for $K=\mathbb Q(i)$ or $K= \mathbb Q(\sqrt{-7})$, the proof of Theorem \ref{main} in those cases is complete.

Let us now focus on $K= \mathbb Q(\sqrt{-2})$. As previously noted, the total dimension of the cuspidal space is $6$, there are only $2$ newforms at this level. On the levels $\mathfrak a^{i}$, $1 \leq i \leq 7$ there is only newform at level $\mathfrak a^5$, where $\mathfrak a$ is the only prime ideal above $2$.  All these forms have integer Hecke eigenvalues.


At the moment of writing this preprint LMFDB \cite{lmfdb} did not include a database for Bianchi modular forms, so the author computed the eigenvalues independently using the available Magma packages. The results match perfectly with the ones that were since then made available in the LMFDB by John Cremona. 

In the spirit of Langland's philosophy, there should be a motif attached to the cuspidal newforms above. This motif is not always an elliptic curve, as one can see in \cite[Theorem 5]{cremonatwist} and \cite{AFreitasSiksek}*{Section 4}. We prove Theorem 1.1 here by first proving that these newforms correspond to elliptic curves, but as the referees pointed out, one could finish now the proof of this theorem via the following remark.

\begin{remark}  The Bianchi newforms above do not depend on the exponent $p$ nor on the unknown solution $(a,b,c) \in \mathcal O_K^3$ of \eqref{fermeq}. 
	If one of them comes from a solution to \eqref{fermeq} via the above discussion, then its Hecke eigenvalue at $(\sqrt{-2}-1) \subseteq \mathcal O_K$ is congruent modulo $p$ to $\Tr \overline{\rho}_{E,a,b,c}(\Frob_{(\sqrt{-2}-1)})$ and one obtains an upper bound on $p$ which completes the proof of Theorem 1.1. The reader interested mostly in the main result can directly consult the relevant numerical values on the last page of this article. 
\end{remark}

We say that an elliptic curve $C/K$ of conductor $\mathcal N \subseteq \mathcal O_K$, corresponds to a cuspidal Bianchi modular form $F$ for $\Gamma_0(\mathcal N)$ if the $L$-series $L(C/K,s)$ is the Mellin transform $L(F,s)$ of $F$.

If $C/K$ doesn't have complex multiplication and in addition $L(C,s)$ together with its character twists have analytic continuation and satisfies the functional equation, then by results of \cite{jacquet} it follows that there exists a Bianchi modular form for $\Gamma_0(\mathcal N)$, that is a newform, such that its
	Mellin transform is equal to $L(C,s)$.

We will make all of the above explicit for our curves.
Looking in the LMFDB database we found three elliptic curves that are good candidates for such a correspondence with our three newforms. By checking the traces of Frobenius at the first few primes, we notice that the curve

$$E_1 : y^2=x^3+x,$$
with $j(E_1)=1728$ and LMFDB label 2.0.8.1-32.1-a2 should correspond to the newform at level $\mathfrak a^5$, i.e. the Bianchi modular form 2.0.8.1-32.1-a, if we want to stay in the language of the LMDFDB.

Similarly,
$$E_{2} :  y^2=x^3+x^2+x+1,$$
with  $j(E_2)=128$ and LMFDB label 2.0.8.1-256.1-a1 should correspond to the Bianchi modular form 256.1-a and
$$E_3 : y^2=x^3-x^2+x-1,$$
with $j(E_3)=128$ and LMFDB label 2.0.8.1-256.1-b1, should correspond to the Bianchi modular form 256.1-b.

Since the curves $E_1, E_2$ and $E_3$ are base changes from elliptic curves defined over $\mathbb Q$, by the celebrated modularity theorem \cite{modularity2001}, we know that their $L$-functions have analytic continuation.

In this very fortunate situation, we can establish the desired connection between these curves defined over $K$ and the Bianchi modular forms mentioned above. Instead of using the Faltings-Serre method, we are going to use the theory of lifting classical modular forms from $\mathbb Q$ to forms of imaginary quadratic fields with discriminant. A summary of the theory in this special case is given in \cite{cremonatwist}*{Section 4}.

An aspect of this theory worth mentioning is the following.  Suppose that $f$ is a cuspidal modular form (in the classical sense) such that $f$ does not have complex multiplication by $K$ (i.e. $f \otimes \chi \neq f$, where $\chi$ is the quadratic character attached to $K/ \mathbb Q$). Then, the lift $F$ of $f$ is cusp form for $K$ for $\Gamma_0( \mathfrak N)$, where $\mathfrak N$ is an ideal of $\mathcal O_K$ supported only on the primes of $K$ that divide $D$ (the discriminant of the quadratic extension) and the level of the cusp form $f$ we started with. If $F$ is the lift of a newform $f$, then it is also the lift of $f \otimes \chi$ and of no other cuspidal newform over $\mathbb Q$.

In terms of $L$-series we have
$$L(F,s)= L(f,s) L(f \otimes \chi, s).$$

\begin{lemma} Let $K$ be a quadratic imaginary field of discriminant $-D$ and $\chi$ the Dirichlet character associated to $K/\mathbb Q$. Suppose $f \in S_2(\Gamma_0(N_1))$ is a classical cuspidal newform of level $N_1 \in \mathbb Z_{\geq 1}$ such that $f \otimes \chi \in S_2(\Gamma_0(N_2))$ is classical cupidal form of level $N_2 \in \mathbb Z_{\geq 1}$, $f \neq f \otimes \chi$. Then $f$ and $f \otimes \chi$ lift to the same form $F$ for $K$ with level an ideal of norm $N_1N_2/D^2$.
	\end{lemma}

\begin{proof} We only have to prove the statement about the norm of level of $F$ and this is a relatively easy exercise which makes use of the functional equations.
	
	The completed $L$ functions $\Lambda(f,s)= (2 \pi)^{-s} N_1^{s/2} \Gamma(s) L(f,s)$ and $\Lambda(f \otimes \chi,s) = (2 \pi)^{-s} N_2^{s/2} \Gamma(s) L(f \otimes \chi,s)$ satisfy functional equations
	$$\Lambda(f,s)= \pm \Lambda(f,2-s) \text{ and } \Lambda(f \otimes \chi,s)= \pm \Lambda(f \otimes \chi,2-s).$$
	It is known (see, for example, page 414 of \cite{cremonatwist} or the discussion in \cite{crewhit}) that the Mellin transform of $F$,  $$\Lambda(F,s)=D^{s-1}(2 \pi)^{-2s} \Gamma(s) L(F,s)$$ is an entire function of $s$ with functional equation
	\begin{equation} \label{fcteq} \Lambda(F,s) = \pm N(\mathfrak N)^{1-s} \Lambda(F,s),\end{equation}
	where $\mathfrak N \subset \mathcal O_K$ is the level of the newform $F$. 
	
	One can see that
	$$\Lambda(F,s)= D^{s-1} (N_1N_2)^{-s/2} \Lambda(f,s) \Lambda(f \otimes \chi,s).$$
	Now if we make the substitution $s \leftrightarrow 2-s$ and use the functional equations we get
	$$\Lambda(F,2-s)= \pm D^{1-s}(N_1N_2)^{\frac{s-2}{2}} \Lambda(f,s) \Lambda(f \otimes \chi,s) \Leftrightarrow$$
	$$\Lambda(F,2-s)= \pm D^{1-s}(N_1N_2)^{\frac{s-2}{2}} \cdot (N_1N_2)^{s/2} D^{1-s} \Lambda(F,s).$$
	Rearranging the terms, we get
	$$\Lambda(F,s)= \pm \left[ \left( \frac{N_1N_2}{D^2} \right)  \right]^{1-s} \Lambda(F,2-s)$$
	and comparing with \eqref{fcteq} we finish the proof. 
	
	\end{proof}

We should remind the reader that when we say that an elliptic curve corresponds to a newform both classical or over $K$, we actually mean that they have matching $L$-functions.

The rational elliptic curve $E_1/ \mathbb Q$ corresponds to the cuspidal modular form
$f_1(z) = q -2q^5 - 3q^9 +O(q^{10})$
of weight $2$ and level $32$ and coefficient field $\mathbb Q$.

If we let $\chi$ be the character associated to the quadratic imaginary extension $K/ \mathbb Q$ and twist $f_1$ by this character, we get
$f_1 \otimes \chi(z) = q +2q^5 -3q^9 +O(q^{10}),$
a rational cuspidal modular form of the same weight. The level of $f_1 \otimes \chi$ is $64$ and its $L$-function matches the one of the rational elliptic curve $E_1^{\chi}$ (the twist of $E_1$ by $\chi$). 

Using the previous lemma, we know that the holomorphic cusp forms $f_1$ and $f_1\otimes \chi$ lift to $K$ the same newform $F_1$ over $K$ of level with norm $32$. But there is just one ideal of norm $32$ in $\mathcal O_K$, namely $\mathfrak a^5$. We computed that there is only one newform at this level, so this must be our $F_1$. Thus,

%
%
$$L(F_1,s)= L(f_1,s)L(f_1 \otimes \chi,s),$$
By definition,
$$L(E_1/K,s)=L(E_1,s)\cdot L(E_1^{\chi},s),$$
therefore
$$L(E_1/K,s) = L(f_1,s) \cdot L(f_1 \otimes \chi,s)=L(F,s).$$

In exactly the same way, we prove that $E_2$ and $E_3$ correspond to newforms with level of norm $256$. Therefore, the level of these newforms is $\mathfrak a^8$. There are 2 newforms at this level, and they have LMFDB  label 2.0.8.1-256.1-a and 2.0.8.1-256.1-b. By looking at the first few Hecke eigenvalues, we observe that $E_2$ corresponds to the 1-a Bianchi modular form and $E_3$ corresponds to the other. 

\begin{definition}
	Given two elliptic curves $E$ and $E'$, defined over a number field $K$,  and some rational prime number $p$, we write that
	$E \sim_p E'$ if their corresponding mod $p$ Galois representations $\overline{\rho}_{E}, \overline{\rho}_{E'}: G_K \to \GL2(\mathbb F_p)$ are isomorphic. 
	\end{definition}

Denote by $\overline{\rho}_{E_i}$ the mod $p$ Galois representations associated to $E_i$ for all $1 \leq i \leq 3$.  
So far, we have used Conjecture \ref{serreconj} to prove that from $(a,b,c) \in \mathcal O_K$, a non-trivial solution to \ref{fermeq}, one obtains an elliptic curve $E_{a,b,c}$. The mod $p$ Galois representation of this curve has the property there exists $i \in \{1,2,3 \}$  such that for almost all prime ideals  $(\pi) \subseteq \mathcal O_K$,
$\Tr \overline{\rho}_{E,a,b,c}(\Frob_{(\pi)}) = \Tr \overline{\rho}_{E_i}(\Frob_{(\pi)}).$

Since the Frobenius elements are dense and $\overline{\rho}_{E,a,b,c}$ is irreducible, Brauer-Nesbitt theorem gives that $\overline{\rho}_{E,a,b,c}$ and $\overline{\rho}_{E_i}$ are isomorphic. 

In order to finish the proof of Theorem \ref{main}, we have to eliminate all of the following possibilities
\begin{equation} \label{curve1} E=E_{a,b,c} \sim_p E_1,\end{equation}
\begin{equation} \label{curve2} E=E_{a,b,c} \sim_p E_2, \end{equation}
or that 
\begin{equation} \label{curve3} E = E_{a,b,c} \sim_p E_3 \end{equation}

The huge advantage now is that the curves $E_1$, $E_2$ and $E_3$ do not depend on the exponent $p$ nor on the unknown solution $(a,b,c) \in \mathcal O_K^3$ of \eqref{fermeq}.

It's very easy to see that \eqref{curve1} can't happen. This is because $E_1$ has complex multiplication and this imposes extra restrictions on its mod $p$ representation $\overline{\rho}_{E_1}.$ In particular, the latter is not surjective, but we proved that $\overline{\rho}_{E,a,b,c}$ is surjective.

Now we restrict our attention to \eqref{curve2} and \eqref{curve3}. Notice that the ideal $(3) \mathcal O_K$ splits as $(3)=(\sqrt{-2} + 1 )(\sqrt{-2} -1)$. We denote by $\mathfrak m_1 = (\sqrt{-2} + 1 )$. 

Suppose that $E=E_{a,b,c}$ has good reduction at $\mathfrak m_1$. The fact that $E/K$ has full $2$-torsion implies that $4 | \# E(\mathcal O_K / \mathfrak m_1)$. The Hasse bounds  $1 \leq  \#E(\mathcal O_K / \mathfrak m_1) \leq 7$ imply that $\# E(\mathcal O_K / \mathfrak m_1)=4$.

The curves $E_2$ and $E_3$ have good reduction at $\mathfrak m_1$ and $\# E_2(\mathcal O_K / \mathfrak m_1) = 6$ and $\# E_3(\mathcal O_K / \mathfrak m_1)=2$.
If \eqref{curve2} holds, then $p | \# E_2(\mathcal O_K / \mathfrak m_1) - \# E(\mathcal O_K / \mathfrak m_1)=2$. Similarly, if \eqref{curve3} is true, then $p | \# E_3(\mathcal O_K / \mathfrak m_1) - \# E(\mathcal O_K / \mathfrak m_1)=-2$. Both lead us to contradictions with the size of the exponent $p$.

Suppose that $E$ has multiplicative reduction at $m_1$. Looking at the traces of Frobenius, we get that
$$\left\{\begin{array}{l} \eqref{curve2} \Longrightarrow \pm(3+1) = 3+1-6 \pmod{p} \text{ and } \\
\eqref{curve3} \Longrightarrow \pm(3+1)= 3+1-2 \pmod{p}
  \end{array} \right.$$
  both leading to contradictions to our assumption that $ p \geq 19$.


\bibliographystyle{amsplain}
\bibliography{perf-pow}

\end{document}